\documentclass[12pt]{amsart}
\usepackage{amsmath,amsfonts,amsthm,amssymb}
\usepackage{graphicx}
\usepackage{mathrsfs}

\newcommand{\R}{\mathbb{R}}

\newcommand{\ie}{{\em i.$\,$e.$\!$} }
\newcommand{\e}{\varepsilon}

\newtheorem{theorem}{Theorem}[section]

\newtheorem{lemma}[theorem]{Lemma}

\theoremstyle{definition}
\theoremstyle{remark}
\newtheorem{remark}{Remark}[section]%[subsection]

\DeclareMathOperator{\supp}{supp}

\DeclareMathOperator{\dist}{dist}

\numberwithin{equation}{section}

\makeindex

\begin{document}

\title[Weak-type (1,1)]{Weak-type $(1,1)$ estimates for strongly singular 
operators}

\author{Magali Folch-Gabayet}
\address{Instituto de Matem\'aticas\\
Universidad Nacional Aut\'onoma de M\'exico\\
Circuito Exterior, Cd. Universitaria\\Mexico City, Mexico 04510}
\email{folchgab@matem.unam.mx}

\author{Ricardo A. S\'aenz}
\address{Facultad de Ciencias\\
Universidad de Colima\\
Ave. Bernal D\'iaz del Castillo \# 340, Col. Villa San Sebasti\'an\\
Colima, Colima, Mexico 28045}
\email{rasaenz@ucol.mx}
\date{\today}
\subjclass[2010]{42B20; 44A25}
\keywords{singular integrals, oscillatory integrals, strongly singular
operators, weak-type estimates}

\begin{abstract}
Let $\psi$ be a positive function defined near the origin such that
$\lim_{t\to 0^{+}}\psi(t)=0$. We consider the operator
\begin{equation*}
T_\theta f(x) = \lim_{\e\to 0^+} \int_\e^1 e^{i\gamma(t)}f(x-t)
\frac{dt}{t^{\theta}\psi(t)^{1-\theta}},
\end{equation*}
where $\gamma$ is a real function with $\lim_{t\to 0^+}|\gamma(t)| = \infty$ 
and $0 \le \theta \le 1$. Assuming certain regularity and growth conditions on
$\psi$ and $\gamma$, we show that $T_1$ is of weak type $(1,1)$.
\end{abstract}

\maketitle

%%%%%%%%%%%%%%%%%%%%%%%%%%%%%%%%%%%%%%%%%%%%%%%%%%%%%%%%%%%%%%%%%%%
%%%%%%%%%%%%%%%%%%%%%%%%%%%%%%%%%%%%%%%%%%%%%%%%%%%%%%%%%%%%%%%%%%%
\section{Introduction and preliminaires}
%%%%%%%%%%%%%%%%%%%%%%%%%%%%%%%%%%%%%%%%%%%%%%%%%%%%%%%%%%%%%%%%%%%
%%%%%%%%%%%%%%%%%%%%%%%%%%%%%%%%%%%%%%%%%%%%%%%%%%%%%%%%%%%%%%%%%%%

Define, for functions $f\in C_{0}^{\infty}(\R)$, the operator
\begin{equation*}
T_{\alpha,\beta}f(x) = \lim_{\e\to 0^{+}}
\int_\e^1 e^{it^{-\alpha}} f(x - t) \frac{dt}{t^\beta},
\end{equation*}
where $\alpha > 0$ and $\beta \ge 1$. The following theorem was proved by
Hirschman in \cite{Hirschman}, and by Fefferman and Stein in \cite{FS72}.

\begin{theorem}\label{Hirs-thm}
Let $\alpha>0$ and $\beta \ge 1$. Whenever $\alpha + 2\geq 2\beta$, the
following hold.
\begin{enumerate}
\item $T_{\alpha, \beta}$ extends to a bounded operator on $L^{2}(\R)$.
\item If $|\frac{1}{2}-\frac{1}{p}|\leq \frac{1}{2}-\frac{\beta-1}{\alpha}$
then $T_{\alpha, \beta}$ extends to a bounded operator on $L^{p}(R)$ for
$1 < p <\infty$.
\item If $|\frac{1}{2}-\frac{1}{p}|>\frac{1}{2}-\frac{\beta-1}{\alpha}$ then
$T_{\alpha, \beta}$ is not a bounded operator on $L^{p}(\R)$.
\end{enumerate}
\end{theorem}

The case $p=1$ was treated by Fefferman in \cite{Fefferman70} where he proves
the following.

\begin{theorem}\label{Feffermanthm}
$T_{1,1}$ is of weak type $(1,1)$.
\end{theorem}

Cho and Yang in \cite{ChoYang} considered the operators
\begin{equation*}
T_{\lambda, \alpha,\beta}f(x) = \lim_{\e\to 0^{+}}
\int_\e^1  e^{i\lambda t^k} e^{it^{-\alpha}} f(x - t) \frac{dt}{t^\beta},
\end{equation*}
where $k\ge 2$ is an integer, and obtained estimates for the $L^2$ norm of
$T_{\lambda, \alpha,\beta}$ as $\lambda\to\infty$. Namely that 
$||T_{\lambda,\alpha,\beta}||\approx\lambda^{-(\alpha/2-\beta+1)/(\alpha+k)}$
when $k$ is even, and 
$||T_{\lambda,\alpha,\beta}||\approx\lambda^{-(\alpha/3-\beta+1)/(\alpha+k)}$
when $k$ is odd.

In the present work, we are interested in proving the analogue of Theorem 
\ref{Feffermanthm} when the oscillating factor is worse than a power.
Say, the oscillation could be $e^{1/t}$. Such operators were studied by the
first author in \cite{Folch99}. Given $0\leq \theta\leq 1$, let
\begin{equation}\label{Ttheta-oper}
T_{\theta}f(x)=\lim_{\epsilon\to
0^{+}}\int_{\epsilon}^{1}e^{i\gamma(t)}f(x-t)\frac{t^{-\theta}}{\psi(t)^{1-\theta}}dt, 
\end{equation}
where the functions $\gamma\in C^3((0,1])$ and $\psi\in C^2([0,1])$ satisfy
the following assumptions, for some $t_0, s_0, C>0$:
\begin{enumerate}
\item[(a.1)] $\gamma$, $\psi$ and their derivatives $\gamma'$, $\gamma''$ and
$\psi'$, are all monotone. We also assume $\gamma'(t) > 0$, with $\gamma'$
decreasing on $(0,1]$.
\item[(a.2)] For $0 < t < t_0$, 
\begin{equation}\label{der-ratios}
\Big| \frac{\psi'(t)}{\psi(t)} \Big|\le \frac{1}{2} 
\Big| \frac{\gamma'''(t)}{\gamma''(t)} \Big|.
\end{equation}
\item[(a.3)] For $s > s_0$, 
\begin{equation*}
|\gamma''(\gamma'^{-1}(2s))| \le C |\gamma''(\gamma'^{-1}(s))|.
\end{equation*}
\item[(a.4)] There exist $\epsilon > 0$, $A > 1 + \epsilon$ such that
\begin{equation*}
\gamma'(t) \ge A \gamma'((1+\epsilon)t),
\end{equation*}
for $0 < t < t_0$.
\item[(a.5)] There exists $1/2 < \lambda < 1$ such that
\begin{equation*}
|\gamma''(t)| \le C \gamma'(t)^{2\lambda},
\end{equation*}
for $0 < t < t_0$.
\end{enumerate}

The functions $\gamma(t) = -t^{1-\sigma}$ and $\psi(t) = t^{(\sigma+1)/2}$, for
$\sigma>1$, satisfy the previous assumptions, and correspond to the operator
$T_{\alpha,\beta}$ above with  $\alpha = \sigma - 1$ and 
$\beta = (\sigma + 1)/2 = (\alpha + 2)/2$.

The functions $\gamma(t) = e^{1/t}$ and $\psi(t) = t^{3/2}e^{-1/2t}$ also 
satisfy assumptions (a.1)-(a.5), and in this case $\psi$ is infinitely flat at
the origin.

\begin{remark}\label{inverse}
Assumption (a.1) implies the existence of the inverse $\gamma'^{-1}(s)$ of
$\gamma'(t)$ for $s\ge \gamma'(1)$, a fact that has been used in assumption
(a.3).
\end{remark}
\begin{remark}\label{necess-est}
The estimate \eqref{der-ratios} of assumption (a.2) implies
\begin{equation*}
\frac{1}{\psi(t)} \le C' \sqrt{|\gamma''(t)|}
\end{equation*}
for $0 < t < t_0$ and a constant $C'$. This estimate, as discussed in
\cite{Folch99}, is necessary for the operator $T_0$ to be bounded on $L^2(\R)$.
\end{remark}
\begin{remark}\label{reciprocal-est}
From assumption (a.4), for $0 < t < t_0$, $1/t < \gamma'(t)$. Hence, for such
$t$, $t < \gamma'^{-1}(1/t)$, as $\gamma'$ is decreasing.
\end{remark}
\begin{remark}\label{reciprocal-ratio-est}
Assumption (a.4) also implies the estimate
\begin{equation*}
|\gamma''(t)| \ge C'' \frac{\gamma'(t)}{t},
\end{equation*}
for a constant $C''$ and $0 < t < t_0$.
\end{remark}

The previous remarks were stated and verified in \cite{Folch99}, where the
following theorem is proved.

\begin{theorem}\label{Ttheta-thm}
Suppose $\gamma$ and $\psi$ satisfy assumptions \textnormal{(a.1)-(a.5)}. Then
\begin{enumerate}
\item $T_\theta$ is a bounded operator on $L^{p}(\R)$ for 
\begin{equation*}
\frac{1}{p} = \frac{1 + \theta}{2}
\end{equation*}
and $0\leq\theta < 1$ , and the operator norm $||T_{\theta}||_{L^{p}\to L^{p}}$
depends only on $\theta$.
\item $T_{1}$ is a bounded operator from $H^{1}(\R)$ to $L^{1}(\R)$.
\end{enumerate}
\end{theorem}

In this work we prove the following theorem.

\begin{theorem}\label{weak-type11}
Under assumptions \textnormal{(a.1)-(a.5)}, $T_{1}$ is of weak type $(1,1)$.
\end{theorem}

Even though we have a singularity of the form $1/t$, one cannot apply standard
Calder\'on-Zygmund theory for the operator $T_1$ because of the oscillating
factor $e^{i\gamma(t)}$. Chanillo and Christ in \cite{ChanilloChrist87} proved
weak type $(1,1)$ estimates for operators with oscillations of the form
$e^{iP(t)}$, where $P$ is a polynomial. Later, Folch-Gabayet and Wright
\cite{FGWright12} considered oscillations of the form $e^{iR(t)}$, where $R$ is
a rational function. 

For the proof of Theorem \ref{weak-type11}, we consider, for $\e>0$ and 
$\beta\in\R$, the kernel $K_{\e,\beta}$ given by
\begin{equation}\label{kernel}
K_{\e, \beta}(x) =
\begin{cases}
\dfrac{e^{i\gamma(x)}}{x^{1+i\beta}\psi(x)^{-i\beta}} & \e \le x \le 1\\
0 & \text{otherwise,}
\end{cases}
\end{equation}
where the functions $\gamma$ and $\psi$ satisfy assumptions (a.1)-(a.5). In
\cite{Folch99}, the following properties of $K_{\e,\beta}$ were proved.

\begin{lemma}\label{K-est}
There exists a constant $C$, independent of $\e$ and $\beta$, such that,
for $\xi\in\R$, $|\xi|>1$,
\begin{equation*}
|\widehat{K_{\e,\beta}}(\xi)| \le 
\frac{C(1+|\beta|)}{\sqrt{|\gamma''(\gamma'^{-1}(|\xi|))|}\gamma'^{-1}(|\xi|)},
\end{equation*}
and, for $|\xi|\le 1$,
\begin{equation*}
|\widehat{K_{\e,\beta}}(\xi)| \le C(1+|\beta|).
\end{equation*}
\end{lemma}

\begin{lemma}\label{K-diff}
There exist $C>0, \eta>0$ such that, for $y\in\R$, $|y|<\eta$,
\begin{equation}\label{K-singint}
\int_{|x|\ge 2\gamma'^{-1}(1/|y|)} 
|K_{\e,\beta}(x - y) - K_{\e,\beta}(x)| dx \le C(1 + |\beta|).
\end{equation}
\end{lemma}

Lemma \ref{K-est} follows from stationary phase methods, while Lemma
\ref{K-diff} follows from explicit estimates on the difference in the integral.
Both proofs heavily use the assumptions (a.1)-(a.5). See \cite{Folch99} for
details.

%%%%%%%%%%%%%%%%%%%%%%%%%%%%%%%%%%%%%%%%%%%%%%%%%%%%%%%%%%%%%%%%%%%
%%%%%%%%%%%%%%%%%%%%%%%%%%%%%%%%%%%%%%%%%%%%%%%%%%%%%%%%%%%%%%%%%%%
\section{Proof of the main theorem}
%%%%%%%%%%%%%%%%%%%%%%%%%%%%%%%%%%%%%%%%%%%%%%%%%%%%%%%%%%%%%%%%%%%
%%%%%%%%%%%%%%%%%%%%%%%%%%%%%%%%%%%%%%%%%%%%%%%%%%%%%%%%%%%%%%%%%%%

In the rest of this paper we will always assume (a.1)-(a.5). Theorem
\ref{weak-type11} will be a consequence of the following  theorem.

%%%%%%%%%%%%%%%%%%%%%%%%%%%%%%%%%%%%%%%%%%%%%%%%%%%%%% Theorem
\begin{theorem}\label{weaktype-thm}
There exists a constant $A$, independent of $\e$ and $\beta$, such that, for
all $f\in L^1(\R)$ and $\alpha>0$, 
\begin{equation}\label{weaktype-est}
|\{ x\in\R: |T_{\e,\beta}f(x)|>\alpha \}| \le \frac{A(1+|\beta|)}{\alpha}||f||_{L^1},
\end{equation}
where $T_{\e,\beta}$ is the convolution operator 
\[
T_{\e,\beta}f = K_{\e,\beta}*f.
\]
\end{theorem}
%%%%%%%%%%%%%%%%%%%%%%%%%%%%%%%%%%%%%%%%%%%%%%%%%%%%%% Theorem

As estimate \eqref{weaktype-est} is uniform in $\e$ and $\beta$, Theorem
\ref{weak-type11} follows by taking $\e\to 0$ and $\beta=0$.

The proof of Theorem \ref{weaktype-thm} will use the following extension of a
standard Whitney decomposition (cf. \cite{Stein}). 

%%%%%%%%%%%%%%%%%%%%%%%%%%%%%%%%%%%%%%%%%%%%%%%%%%%%%% Lemma
\begin{lemma}\label{altCZ}
Let $\Omega\subset\R$ be open and $F=\R\setminus\Omega$. Then there exists a
collection of intervals $\{I_k\}$ with disjoint interiors and two constants
$C\ge c >3$,  such that 
$\Omega = \bigcup I_k$ and
\[
c \gamma'^{-1}(1/|I_k|) \le \dist(I_k,F) \le C \gamma'^{-1}(1/|I_k|).
\]
\end{lemma}
%%%%%%%%%%%%%%%%%%%%%%%%%%%%%%%%%%%%%%%%%%%%%%%%%%%%%% Lemma

Note that the distance of each interval to the complement of $\Omega$
is estimated in terms of $\gamma'^{-1}$, rather than just to its length.

%%%%%%%%%%%%%%%%%%%%%%%%%%%%%%%%%%%%%%%%%%%%%%%%%%%%%% Proof
\begin{proof}
Let $\mathscr M_k$ be the mesh of dyadic intervals of length $2^{-k}$ in 
$\R$. For a number $a>0$ that will be fixed later, let
\[
\Omega_k = \{x\in\Omega: a \gamma'^{-1}(2^{k+1}) \le \dist(x,F)
\le a \gamma'^{-1}(2^k) \}.
\]
If $I\in\mathscr M_k$ is such that $I\cap \Omega_k\not=\emptyset$, then
\[
\dist(I,F) \le a \gamma'^{-1}(2^k) = a \gamma'^{-1}(1/|I|).
\]

Now, from assumption (a.4), $\gamma'(t) \ge A \gamma'((1+\epsilon)t)$. If 
$l\ge1$ is such that $4>A^l\ge 2$, then $\gamma'(t) \ge 2 
\gamma'((1+\epsilon)^lt)$, and thus
\[
\gamma'^{-1}(2^k) \le (1+\epsilon)^l \gamma'^{-1}(2^{k+1}).
\]
Hence
\[
\begin{split}
\dist(I,F) &\ge a \gamma'^{-1}(2^{k+1})  - |I|
\ge \frac{1}{(1+\epsilon)^l} \gamma'^{-1}(2^k) - |I|\\
&> \Big( \frac{1}{(1+\epsilon)^l} - 1 \Big) \gamma'^{-1}(2^k),
\end{split}
\]
since $|I| < \gamma'^{-1}(1/|I|)$ and $|I| = 2^{-k}$. Therefore, if we set
$a = 5(1+\epsilon)^l$, then we have
\[
4 \gamma'^{-1}(1/|I|) \le \dist(I,F) \le 20 \gamma'^{-1}(1/|I|),
\]
because $a < 5 A^l < 20$.

As $\bigcup \Omega_k = \Omega$, the lemma follows by taking $\{I_k\}$ as the
collection of maximal intervals as above.
\end{proof}
%%%%%%%%%%%%%%%%%%%%%%%%%%%%%%%%%%%%%%%%%%%%%%%%%%%%%% Proof

\begin{remark}\label{finite-int}
If for each $I_k$ as above we define the interval
\begin{equation}\label{intervals}
I_k^{*} = [y_k  - 3\gamma'^{-1}(1/|I_k|), y_k  + 3\gamma'^{-1}(1/|I_k|)],
\end{equation}
where $y_k$ is the center of $I_k$, then there exists a fixed $N$ such that at
most $N$ intervals $I_j^{*}$ intersect $I_k^{*}$. Indeed, if $x\in I_k^{*}$,
then
\[
\gamma'^{-1}(1/|I_k|) \le \dist(x,F) \le 23 \gamma'^{-1}(1/|I_k|).
\]
Note that $|I_k^{*}| = 6\gamma'^{-1}(1/|I_k|) \ge 6/23 \dist(x,F)$, and that
$I_k^{*}$ is contained in an interval of length
$12\gamma'^{-1}(1/|I_k|) \le 12 \dist(x,F)$ with center $x$. Hence, there can
be at most 
\[
\frac{12 \dist(x,F)}{6/23 \dist(x,F)} = 46
\]
such intervals, so we can take $N=46$.
\end{remark}

\begin{remark}\label{unionIk}
From the discussion in Remark \ref{finite-int}, we see that
$I_k^* \subset \Omega$ and, if $50I_k^*$ is the interval with the same center
as $I_k^*$ with 50 times its length, then $50I_k^*\setminus\Omega \not=
\emptyset$.

Hence $\bigcup I_k^* = \Omega$ and, since each $I_k^*$ intersects at most a
finite fixed number of other such intervals,
\begin{equation}\label{sum-lengths}
\sum_k |I_k^*| \lesssim |\Omega|.
\end{equation}
\end{remark}

%%%%%%%%%%%%%%%%%%%%%%%%%%%%%%%%%%%%%%%%%%%%%%%%%%%%%% Proof of Theorem
\begin{proof}[Proof of Theorem \ref{weaktype-thm}]
We follow the idea of the proof of Theorem 2' of \cite{Fefferman70}, but we now
apply Lemma \ref{altCZ} with
\[
\Omega = \{ x\in\R: Mf(x) > \alpha' \},
\]
where $\alpha' = \alpha/(1 + |\beta|)$ and $Mf$ is the Hardy-Littlewood maximal
function of $f$. Let $f = g + b = g + \sum b_k$, where each
\[
b_k(x) = \Big(f(x) - \frac{1}{|I_k^*|}\int_{I_k^*}f\Big)\chi_{I_k^*}(x)
\]
and $I_k^*$ is as in Remark \ref{finite-int}.
Note that $\supp b_k \subset I_k^*$, $\int b_k = 0$, and we have the estimate 
\begin{equation*}
\int |b_k| \le 2\int_{I_k^*}|f| \le 100|I_k^*|\cdot\frac{1}{|50I_k^*|}
\int_{50I_k^*}|f|  \lesssim \alpha' |I_k^{*}|
\end{equation*}
because $50I_k^*\setminus\Omega \not= \emptyset$, by Remark \ref{unionIk}.
This implies, by \eqref{sum-lengths},
\begin{equation*}
\int |b| \le \sum_k \int |b_k|\lesssim \alpha' \sum_k |I_k^*| \lesssim
\alpha' |\Omega| \lesssim \int |f|,
\end{equation*}
by the Hardy-Littlewood maximal theorem.

As usual, $|g(x)| \le\alpha'$ if $x\not\in\Omega$ (by the Lebesgue 
differentiation theorem) and, if $x\in I_k^*$,
\begin{equation*}
|g(x)| \le \frac{1}{|I_k^*|}\int_{I_k^*}|f| \lesssim \alpha',
\end{equation*}
By Lemma \ref{K-est}, $T_{\e,\beta}$ is bounded in $L^2$ with norm
$(1+|\beta|)$, hence
\begin{equation}\label{T2}
\begin{split}
|\{ x\in\R: |T_{\e,\beta}g(x)|>\alpha \}| &\le \frac{1}{\alpha^2}
\int |T_{\e,\beta}g|^2\\
&\lesssim \frac{(1 + |\beta|)^2}{\alpha^2}\int |g|^2 \\
&\lesssim \frac{(1 + |\beta|)^2}{\alpha^2}\alpha'  \int |g|
\lesssim \frac{1+|\beta|}{\alpha}||f||_{L^1}.
\end{split}
\end{equation}

Now, let $\phi\in C^\infty(\R)$ be nonnegative, supported in $\{x:|x|<1\}$ and 
with $\int \phi = 1$. For each $k$, define
\[
\phi_k(x) = \frac{1}{|I_k|} \phi\Big(\frac{x}{|I_k|}\Big),
\]
\ie, $\phi_k$ is the dilation of $\phi$ with scale $|I_k|$. Set
$\tilde b_k = \phi_k * b_k$ and
\[
\tilde b = \sum_{|I_k^*| \le 1} \tilde b_k.
\]
Note that, if $|I_k^*| > 1$ and $x\not\in 3I_k^{*}$, then
\[
K_{\e,\beta}*b_k(x) = \int_{I_k^*} K_{\e,\beta}(x-y)b_k(y) dy = 0,
\]
because $|x-y|\ge |I_k^*| > 1$ and $\supp K_{\e,\beta}\subset (0,1]$.
So, for $x\not\in \bigcup 3I_k^*$, $K_{\e,\beta}*b = \sum_{|I_k^*|\le 1} b_k.$
Thus
\[
K_{\e,\beta}*b(x) - K_{\e,\beta}*\tilde b(x) = \sum_{|I_k^*| \le 1} 
\big( K_{\e,\beta}*b_k(x) - K_{\e,\beta}*\tilde b_k(x) \big)
\]
for $x\not\in 3I_k^{*}$, and if $|I_k^*|\le 1$,
\begin{equation*}
\begin{split}
\int_{\R\setminus 3I_k^{*}}&\big| K_{\e,\beta}*b_k(x) - K_{\e,\beta}*\tilde b_k(x) \big|dx\\
&\le \int_{\R\setminus 3I_k^{*}} \Big| 
\int_{I_k^*} (K_{\e,\beta}(x-y) - K_{\e,\beta}*\phi_k(x-y))b_k(y) dy \Big| dx\\
&\le \int_{I_k^*}\int_{\R\setminus 3I_k^{*}}
\big|K_{\e,\beta}(x-y) - K_{\e,\beta}*\phi_k(x-y)\big| dx \; |b_k(y)|dy \\
&\le \int_{I_k^*}\int_{|z|>2\gamma'^{-1}(1/|I_k|)}
\big|K_{\e,\beta}(z) - K_{\e,\beta}*\phi_k(z)\big| dz \; |b_k(y)|dy,
\end{split}
\end{equation*}
because, if $y\in I_k^*$ and $x\not\in 3I_k^{*}$, then 
\[
|x-y| \ge |I_k^*| > 2\gamma'^{-1}(1/|I_k|).
\]
Also
\begin{multline*}
\int_{|z|>2\gamma'^{-1}(1/|I_k|)}\big|K_{\e,\beta}(z) - K_{\e,\beta}*\phi_k(z)\big| dz\\
= \int_{|z|>2\gamma'^{-1}(1/|I_k|)}\Big|K_{\e,\beta}(z) - \int_{|w| < |I_k|}
K_{\e,\beta}(z - w)\phi_k(w) dw \Big| dz \\
\le \int_{|w| < |I_k|}
\int_{|z|>2\gamma'^{-1}(1/|I_k|)} |K_{\e,\beta}(z) - K_{\e,\beta}(z-w)| dz \phi_k(w) dw \\
\lesssim 1 + |\beta|,
\end{multline*}
because $|w| < |I_k|$ implies $2\gamma'^{-1}(1/|I_k|) > 2\gamma'^{-1}(1/|w|)$,
so we obtain the estimate using \eqref{K-singint}.

Thus
\begin{multline*}
\int_{\R\setminus\bigcup 3I_k^*} 
|K_{\e,\beta}*b(x) - K_{\e,\beta}*\tilde b(x)| dx\\
\lesssim (1 + |\beta|) \sum_{|I_k^*|\le 1} \int_{I_k^*} |b_k(y)| dy \lesssim 
(1 + |\beta|) ||f||_{L^1},
\end{multline*}
and therefore
\[
|\{x\not\in\bigcup 3I_k^*: 
|K_{\e,\beta}*b(x) - K_{\e,\beta}*\tilde b(x)| > \alpha \}| \lesssim
\frac{(1 + |\beta|)}{\alpha} ||f||_{L^1}.
\]
Since 
\begin{equation*}
\Big|\bigcup 3I_k^*\Big| \le 3 \sum_k |I_k^*| \lesssim 
\frac{1 + |\beta|}{\alpha} ||f||_{L^1},
\end{equation*}
it remains to prove the estimate
\begin{equation*}
\Big| \Big\{ x : | K_{\e,\beta}*\tilde 
b(x) \Big| >\alpha \Big\}\Big| \lesssim \frac{1 + |\beta|}{\alpha} 
||f||_{L^1}.
\end{equation*}

Write
\begin{multline*}
K_{\e,\beta}*\tilde b(x) = \\
\sum_{|I_k^*|\le 1}(1 - \chi_{3I_k^*}(x))
 K_{\e,\beta}*\tilde b_k(x)
+ \sum_{|I_k^*| \le 1} \chi_{3I_k^*}(x)K_{\e,\beta}*\tilde b_k(x).
\end{multline*}
It is clear that the second sum is supported in $\bigcup 3I_k^*$, so it is
enough to prove the estimate
\[
\Big| \Big\{ x : \Big| \sum_{\substack{|I_k^*|\le 1\\x\notin 3I_k^{*}}} 
K_{\e,\beta}*\tilde 
b_k(x) \Big| >\alpha \Big\}\Big| \lesssim \frac{1 + |\beta|}{\alpha} 
||f||_{L^1},
\]
which follows from
%%%%%%%%%%%%%%%%% Claim 1
\begin{equation}\label{L1-est}
\Big|\Big| \sum_{\substack{|I_k^*|\le 1\\x\notin 3I_k^{*}}} 
K_{\e,\beta}*\tilde b_k(x)
\Big|\Big|_{L^1} \lesssim (1 + |\beta|)||f||_{L^1}.
\end{equation}

To prove \eqref{L1-est}, we estimate each integral
\[
\int_{x\notin 3I_k^{*}} |K_{\e,\beta}*\tilde b_k(x)| dx.
\]
Using the fact that $\int b_k = 0$ and $\tilde b_k = \phi_k*b_k$, we see that
\begin{multline*}
K_{\e,\beta}*\tilde b_k(x) = K_{\e,\beta}*\phi_k*b_k(x) = \int_{I_k^*} 
K_{\e,\beta}*\phi_k(x-y) b_k(y) dy\\
= \int_{I_k^*} 
(K_{\e,\beta}*\phi_k(x-y) - K_{\e,\beta}*\phi_k(x-y_k)) b_k(y) dy.
\end{multline*}
Thus
\begin{multline} \label{Kbtilde}
\int_{x\notin 3I_k^{*}} |K_{\e,\beta}*\tilde b_k(x)| dx \le\\
\int_{I_k^*} \int_{x\notin 3I_k^{*}}
|K_{\e,\beta}*\phi_k(x-y) - K_{\e,\beta}*\phi_k(x-y_k)| dx
|b_k(y)| dy.
\end{multline}
Now
\begin{multline*}
K_{\e,\beta}*\phi_k(x-y) - K_{\e,\beta}*\phi_k(x-y_k) \\
= \int_{|z| < |I_k|}
\big( K_{\e,\beta}(x - y - z) - K_{\e,\beta}(x - y_k - z) \big) \phi_k(z) dz, 
\end{multline*}
so the inner integral in \eqref{Kbtilde} is estimated by
\begin{equation}\label{Kbtilde-inn}
\int_{|z| < |I_k|}\int_{x\notin 3I_k^{*}}
\big| K_{\e,\beta}(x - y - z) - K_{\e,\beta}(x - y_k - z) \big| dx \phi_k(z) dz.
\end{equation}
Now, if $x\notin 3I_k^{*}$ and $|z| < |I_k|$, 
\[
|x - y_k - z| \ge |x - y_k| - |z| > |I_k^*| - |I_k|
> 2\gamma'^{-1}\Big(\frac{1}{|I_k|}\Big) \ge
2\gamma'^{-1}\Big(\frac{1}{|y - y_k|}\Big),
\]
and thus \eqref{Kbtilde-inn} is estimated by 
\begin{multline*}
\int_{|z| < |I_k|}
\int_{|w| > 2\gamma'^{-1}(\frac{1}{|y - y_k|})}
\big| K_{\e,\beta}(w) - K_{\e,\beta}(w - (y_k - y)) \big|dw \phi_k(z) dz\\
\lesssim (1 + |\beta|)\int_{|z| < |I_k|} \phi_k(z) dz = 1 + |\beta|.
\end{multline*}
Therefore
\[
\int_{x\notin 3I_k^{*}} |K_{\e,\beta}*\tilde b_k(x)| dx \lesssim
(1 + |\beta|)\int|b_k| \lesssim (1 + |\beta|)\alpha' |I_k^*|,
\]
so
\begin{equation*}
\Big|\Big| \sum_{\substack{|I_k^*| \le 1\\
x\notin 3I_k^{*}}} K_{\e,\beta}*\tilde b_k(x)
\Big|\Big|_{L^1} \lesssim (1 + |\beta|)||f||_{L^1}.
\end{equation*}
We have proved \eqref{L1-est}, and thus completed the proof of Theorem 
\ref{weaktype-thm}.
\end{proof}

%%%%%%%%%%%%%%%%%%%%%%%%%%%%%%%%%%%%%%%%%%%%%%%%%%%%%% Proof of Theorem

\section*{Aknowledgements}
This research was supported by CONACYT Grant FORDECYT 265667. The authors
would like to thank the referee for suggestions that lead to the improvement of
this paper.
%%%%%%%%%%%%%%%%%%%%%%%%%%%%%%%%%%%%%%%%%%%%%%%%%%%%%%
%%%%%%%%%%%           References         %%%%%%%%%%%%%
%%%%%%%%%%%%%%%%%%%%%%%%%%%%%%%%%%%%%%%%%%%%%%%%%%%%%%


\begin{thebibliography}{FGW12}

\bibitem[CC87]{ChanilloChrist87}
Sagun Chanillo and Michael Christ, \emph{Weak {$(1,1)$} bounds for oscillatory
  singular integrals}, Duke Math. J. \textbf{55} (1987), no.~1, 141--155.
  \MR{883667}

\bibitem[CY10]{ChoYang}
Chu-Hee Cho and Chan~Woo Yang, \emph{Estimates for oscillatory strongly
  singular integral operators}, J. Math. Anal. Appl. \textbf{362} (2010),
  no.~2, 523--533. \MR{2557706}

\bibitem[Fef70]{Fefferman70}
Charles Fefferman, \emph{Inequalities for strongly singular convolution
  operators}, Acta Math. \textbf{124} (1970), 9--36. \MR{0257819}

\bibitem[FG99]{Folch99}
Magali Folch-Gabayet, \emph{A family of strongly singular operators}, J.
  Austral. Math. Soc. Ser. A \textbf{67} (1999), no.~1, 58--84. \MR{1699156}

\bibitem[FGW12]{FGWright12}
Magali Folch-Gabayet and James Wright, \emph{Weak-type {$(1,1)$} bounds for
  oscillatory singular integrals with rational phases}, Studia Math.
  \textbf{210} (2012), no.~1, 57--76. \MR{2949870}

\bibitem[FS72]{FS72}
C.~Fefferman and E.~M. Stein, \emph{{$H^{p}$} spaces of several variables},
  Acta Math. \textbf{129} (1972), no.~3-4, 137--193. \MR{0447953}

\bibitem[Hir59]{Hirschman}
I.~I. Hirschman, Jr., \emph{On multiplier transformations}, Duke Math. J
  \textbf{26} (1959), 221--242. \MR{0104973}

\bibitem[Ste93]{Stein}
Elias~M. Stein, \emph{Harmonic analysis: real-variable methods, orthogonality,
  and oscillatory integrals}, Princeton University Press, Princeton, NJ, 1993,
  With the assistance of Timothy S. Murphy, Monographs in Harmonic Analysis,
  III. \MR{95c:42002}

\end{thebibliography}
\end{document}